\documentclass[12pt]{amsart}
\usepackage{graphicx} 
\usepackage[table]{xcolor}
\usepackage{amsthm}
\usepackage{hyperref}
\usepackage[capitalise]{cleveref}
\usepackage{todonotes}
\usepackage{tikz}
\usetikzlibrary{shapes.geometric}

\setuptodonotes{size=\tiny	}

\newcommand{\dtwist}[2]{T^{#1}_{#2}} 

\newcommand{\setbuilder}[2]{\{#1 : #2\}}
\newcommand{\enum}[1]{\{#1\}}

\newcommand{\ints}{\mathbb{Z}}

\newcommand{\irange}[3]{#1 \leq #2 \leq #3}

\newcommand{\functionheader}[3]{#1 : #2 \to #3}

\theoremstyle{definition}
\newtheorem{theorem}{Theorem}[section]
\newtheorem{proposition}[theorem]{Proposition}
\newtheorem{lemma}[theorem]{Lemma}

\newtheorem{sublemma}[theorem]{Sublemma}
\newtheorem{corollary}[theorem]{Corollary}
\newtheorem{remark}[theorem]{Remark}
 
\newtheorem{definition}[theorem]{Definition}
\newtheorem{fact}[theorem]{Fact}
\newtheorem{cons}[theorem]{Construction}
\newtheorem{notation}[theorem]{Notation}

\begin{document}

\title[Connectivity in low and medium complexity]{Optimal connectivity results for spheres in the curve graph of low and medium complexity surfaces}
\maketitle{}





\author{Helena Heinonen},
\author{Roshan Klein-Seetharaman},
\author{Minghan Sun}

\begin{abstract}
    Answering a question of Wright, we show that spheres of any radius are always connected in the curve graph of surfaces $\Sigma_{2,0}, \Sigma_{1,3},$ and $\Sigma_{0,6}$, and the union of two consecutive spheres is always connected for $\Sigma_{0, 5}$ and $\Sigma_{1,2}$. We also classify the connected components of spheres of radius 2 in the curve graph of $\Sigma_{0, 5}$ and $\Sigma_{1,2}$.  
\end{abstract}


\section{Introduction}

\subsection{Main results}

Let $\Sigma = \Sigma_{g,n}$ be a connected surface with genus $g$ and $n$ punctures. We define the complexity of $\Sigma$ to be $\xi(\Sigma) = 3g - 3 + n$. We say $\Sigma$ is 
\begin{itemize}
\setlength\itemsep{0.15em}
\item \emph{exceptional} if $\xi(\Sigma)=1$, i.e. $(g,n)\in \{(1,1),(0,4)\}$,
\item \emph{low complexity} if $\xi(\Sigma)=2$, i.e. $(g,n)\in \{(1,2),(0,5)\}$, 
\item \emph{medium complexity} if $\xi(\Sigma)=3$, i.e. $(g,n)\in \{(2,0), (1,3), (0,6)\}$, 
\item \emph{high complexity} if $\xi(\Sigma)\geq 4$. 
\end{itemize}

Let $ \mathcal{C}\Sigma$ be the curve graph of $\Sigma$. For any vertex $c \in  \mathcal{C}\Sigma$ and radius $r$, let $$S_r = S_r(c) = \{a \in \mathcal{C}\Sigma: d(a, c) = r\}$$ be the sphere of radius $r$ about $c$ in $ \mathcal{C}\Sigma$. We will say that a sphere is connected if the induced subgraph is connected. 

The main results to be proved in this paper are as follows:
\begin{theorem}\label{mainPropLowComp}
    Let $\Sigma_{g,n}$ be low complexity. Fix a center $c \in \mathcal{C}\Sigma$. Then for all $r > 0$ we have that $S_r(c) \cup S_{r+1}(c)$ is connected. 
\end{theorem}

\begin{theorem}\label{mainPropMedComp}
        Let $\Sigma_{g,n}$ be medium complexity. Fix a center $c \in \mathcal{C}\Sigma$. Then for all $r > 0$ we have that $S_r(c)$ is connected. 
\end{theorem}

In the low complexity case, we do not understand in general the connected components of $S_r$. However, we can understand the case of $S_2$.  
\begin{definition}\label{defn:nonisolated_vertices}
    Let $\Sigma_{g,n}$ be low complexity. Fix center $c \in \mathcal{C}\Sigma$. Let $S_{r} ^\prime  (c)$ denote the vertices in $S_r (c)$ which are not isolated in $S_r (c) $. 
\end{definition}

\begin{theorem}\label{thm:S2_connected}  
Let $\Sigma_{g,n}$ be low complexity. Fix center $c \in \mathcal{C}\Sigma$. Then $S ^\prime _{2} (c) $ is connected.  
\end{theorem}   

\subsection{Previous results}
The main contribution of this paper is to strengthen the following theorem from \cite{Wright1}.
\begin{theorem}[\cite{Wright1}, Theorem 1.1]\label{previous} For all $r>0$ and connected surface $\Sigma,$
\begin{enumerate}
    \item If $\Sigma$ has high complexity, then $S_r$ is connected. 
    \item If $\Sigma$ has medium complexity, then $S_r\cup S_{r+1}$ is connected.
    \item If $\Sigma$ has low complexity, then $S_r\cup S_{r+1}\cup S_{r+2}$ is connected.
\end{enumerate}
\end{theorem}
Our \cref{mainPropLowComp} and \cref{mainPropMedComp} strengthen the above theorem, thereby answering \cite[Question 1.7]{Wright1}. Our \cref{mainPropLowComp} and \cref{mainPropMedComp} are sharp because $S_r$ is never connected for $r \geq 1$ in low complexity \cite[Corollary 6.12]{Wright1}. 

Our \cref{thm:S2_connected} describes the connected components of $S_2$ in low complexity.

\subsection{Organization of the proof}

In both the low and medium complexity cases for the connectivity of spheres (\cref{mainPropLowComp} and \cref{mainPropMedComp}), we utilize the same proof strategy, as well as the same preliminary results from \cite{Wright1}. Then we modify the paths obtained in \cite{Wright1} in order to stay closer to $S_r$, with the Bounded Geodesic Image Theorem from \cite{MasurMinsky} as our primary tool. 



Our main contribution in the low complexity case (\cref{mainPropLowComp}) is to construct improved ``preliminary paths" (discussed in \cref{prelimPathSection}), and show this adjustment allows the argument to ultimately yield paths contained in two spheres instead of three.

In the medium complexity case (\cref{mainPropMedComp}), Wright's argument included an induction on radius, for which it was crucial to use essentially non-separating curves (\cref{essentially_non_separating}). Since we assume Wright's result, we avoid arguing by induction, so we are able to use curves which fail to be essentially non-separating to produce paths which stay in a single sphere.


We prove \cref{thm:S2_connected} by showing that $S ^\prime  _2 (c) $ naturally has the structure of a $\ints$-bundle over $S_1 (c) $ (which is a copy of the Farey graph). Moreover, the monodromy of this bundle over a Farey triangle in $S_1 (c) $ is translation by $1$. This $\ints$-bundle structure is related to some existing ideas such as a version of the Lantern relation. But as far as we know, this $\ints$-bundle structure has not been recorded in the literature previously, and we expect it to be of independent interest.

\subsection{Motivation}

This paper continues the tradition of examining the relationship between fine and coarse geometry of the curve graph. As an example, the Bounded Geodesic Image Theorem uses coarse information to deduce a precise result about the vertices on geodesics. 

In particular, we can also gain a better understanding of the coarse geometry of the curve graph as a whole by understanding the fine results. This idea is exemplified in \cite{Wright1} where the linear connectivity of the Gromov boundary (coarse) follows from an analysis of the connectivity of $S_r$ (fine). For previous connectivity results and other related work, see \cite{DeadEnds,chaika2019path,Spheres,Gabai1,klarreich2018boundary,LMS,LSconnectivity,RS,SchleimerEnd}.

Our paper also develops techniques to perform constructions directly in the curve graph rather than spaces of lamination or Teichm\"uller space. 

\subsection{Acknowledgements} 
We would like to thank our mentor Alex Wright for his guidance on this paper and acknowledge that this work was supported by NSF grant DMS-2142712.

\section{Subsurface projections and the Bounded Geodesic Image Theorem} 

In this section we introduce one of our key tools, the Bounded Geodesic Image Theorem, and recall some basic facts about subsurface projections.

Let $U$ be a subsurface of $\Sigma$ and $\alpha \in \mathcal{C}\Sigma$. We say curve $\alpha$ \emph{cuts U} if it is not possible to isotope $\alpha$ out of $U$. We define $\mathcal{C} (\Sigma, U)$ to be the subgraph of $\mathcal{C} \Sigma$ whose vertices are all essential non-peripheral curves that cut $U$, and keeping all possible edges. Note that $\mathcal{C} U$ is contained in $\mathcal{C} (\Sigma, U)$. 

Given a subsurface $U$ of $\Sigma$, there exists a subsurface projection map, denoted $\rho_U$, from the set of curves cutting $U$ to finite subsets of curves on $U$. We will want to recall some key facts about $\rho_U$: 
\begin{enumerate}
    \item [(1)] The values of $\rho_U$ are uniformly bounded in diameter.
    \item [(2)] The map $\rho_U$ is 6-Lipschitz i.e. $$d(\rho_U(\alpha), \rho_U(\beta)) \leq 6 d(\alpha, \beta).$$
    \item [(3)] Define 
    $$
 d_U (\alpha, \beta) = diam (\rho_U(\alpha) \cup \rho_U(\beta)). 
$$ It can easily be verified that $d_U$ satisfies the triangle inequality. 
\end{enumerate}

The following theorem is known as the Bounded Geodesic Image Theorem:

\begin{theorem}\cite[Theorem 3.1]{MasurMinsky}\label{BGI}
Let $U$ be a subsurface of $\Sigma$. There exists $M > 0$ such that if $d_U(\alpha, \beta) \geq M$ then every geodesic from $\alpha$ to $\beta$ in $\mathcal{C} \Sigma$ contains a curve not cutting $U$. 
\end{theorem}

From here on, $M$ will refer to the constant required for Theorem \ref{BGI}, which can be taken independent of $\Sigma$ and $U$ \cite{WebbUniformBGI}. 


\section{Low complexity}

Throughout this section, we deal with $\Sigma = \Sigma_{0,5}$. Assume that a center vertex $c \in \mathcal{C}\Sigma_{0,5}$ is fixed and let $S_r = S_r(c)$. 

\subsection{Organization}


The outcome of this section is to prove \cref{mainPropLowComp}. We do so by first taking arbitrary $a\in S_r$ and $b,b'\in S_{r+1}\cap S_1(a)$ and constructing a preliminary path, described in \cref{modifiedPath}, connecting $b$ to $b'$. We then offer \cref{lem:crucial_lemma} to serve a similar function as \cite[Lemma 6.16]{Wright1} to push this path up to $S_r\cup S_{r+1}$ using Dehn twists, by observing that vertices on this preliminary path only enter $S_3(a)$ when they are close to $S_{r-1} \cup S_r$. This adjustment is sufficient in proving the path stays within two consecutive spheres rather than three.

\subsection{Definitions}

\begin{definition}\label{uniqueBacktrackDef}
    A vertex $x \in S_r$ has \emph{unique backtracking} if it has a unique neighbour in $S_{r-1}.$ 
\end{definition}

\begin{definition}\label{noSideStepDef}
    A vertex $x \in S_r$ has \emph{no sidestepping} if it does not have any neighbour in  $S_r$.
\end{definition}

\begin{definition}\label{forwardFacingDef}
    A vertex $x \in S_r$ is \emph{forward facing} if it has unique backtracking and no sidestepping. 
\end{definition}


\subsection{Pentagons in $\mathcal{C}\Sigma_{0,5}$}
It is important to note that $\mathcal{C}\Sigma$ contains no cycles of length 3 or 4 \cite[Lemma 6.1]{Wright1}. Thus we often study paths on $\mathcal{C}\Sigma$ by using pentagons. 
\begin{definition}
    Label the 5 punctures of $\Sigma$ with the elements of $\mathbb{Z}/5\mathbb{Z}$. The 5 tuple of curves $(a_1, a_2, a_3, a_4, a_5)$ is a \emph{pentagon} if for $i\in\mathbb{Z}/5\mathbb{Z}$:
    \begin{enumerate}
        \item[(1)] $a_i$ goes around punctures $i$ and $i+1$, 
        \item[(2)] the intersection number between $a_i, a_{i+1}$ and $a_i, a_{i-1}$ is 2, and 
        \item[(3)] the intersection number between $a_i, a_{i+2}$ and $a_i, a_{i-2}$ is 2. 
    \end{enumerate}
\end{definition}

To obtain a 5-cycle from a pentagon with vertices $(a_1, a_2, a_3, a_4, a_5)$, we can traverse the curves in the following order: $(a_1, a_3, a_5, a_2, a_4)$. We use the following lemmas to find pentagons in $\mathcal{C}\Sigma_{0,5}$.  

\begin{lemma}\cite[Lemma 6.5]{Wright1}\label{6.5inPaper}
    Suppose $a_1, a_3 \in S_{r-1}$ are adjacent. Then there are curves $a_2, a_3, a_5 \in S_r \cup S_{r+1}$ such that $(a_1,a_2,a_3,a_4,a_5)$ is a pentagon.
\end{lemma}

\begin{lemma}\cite[Lemma 6.6]{Wright1}\label{completeAPentagon}
Suppose $a_1 \in S_{r-1}$ and $a_3,a_4\in S_r\cap S_1(a_1)$ have $i(a_3,a_4)=2.$ Then there exist $a_2,a_5\in S_r\cup S_{r+1}$ such that $(a_1,a_2,a_3,a_4,a_5)$ is a pentagon.
\end{lemma}


\subsection{Preliminary path construction}\label{prelimPathSection}


\begin{proposition}\label{modifiedPath}
    Suppose $a\in S_r$ and $b,b'\in S_{r+1}\cap S_1(a).$ Then there exists a path $\gamma$ from $b$ to $b'$ contained in $S_1(a)\cup S_2(a)\cup S_3(a)$ such that the following hold for all vertices $v$ on the path $\gamma$:
    \begin{enumerate}
        \item If $v\in S_3(a),$ then $d(v,(S_{r-1}\cup S_r)\cap S_1(a))\leq 2.$
        \item If $v\in S_1(a),$ then $v\in S_{r+1}.$
    \end{enumerate}
\end{proposition}

First we recall the following lemmas:
\begin{lemma}\cite[Lemma 6.10]{Wright1}\label{6.10inPaper}
For any $a\in \mathcal{C}\Sigma_{0,5}$ and $x\in S_1(a),$ $x$ is forward facing with respect to $a.$
\end{lemma}

\begin{lemma}\cite[Lemma 6.13]{Wright1}\label{6.13inPaper}
For any $a\in \mathcal{C}\Sigma_{0,5},$ $S_1(a)\cup S_2(a)$ is connected.
\end{lemma}

\begin{lemma}\cite[Lemma 6.14]{Wright1}\label{6.14inPaper}
Suppose $x\in S_r$ is forward facing and $y,y'\in S_1(x)\cap S_{r+1}.$ Then there exists a path from $y$ to $y'$ in $(S_{r+1}\cup S_{r+2})\cap B_2(x).$
\end{lemma}

Lemma \ref{6.14inPaper} gives us the following corollary.
\begin{corollary}\label{avoidS1}
Suppose $x_{j-1},x_j,x_{j+1}$ is a path in $S_1(a)\cup S_2(a)$ with $x_j\in S_1(a).$ Then there exists a path from $x_{j-1}$ to $x_{j+1}$ contained in $(S_2(a)\cup S_3(a))\cap B_2(x_j).$
\end{corollary}
\begin{proof}
    This statement is exactly the conclusion of \cref{6.14inPaper} with $a$ as the center, $x=x_j,$ $y=x_{j-1},$ and $y'=x_{j+1},$ so we only need to check the conditions are satisfied. 

    First we see $x_j$ is forward facing with respect to $a$ because $x_j\in S_1(a)$ by assumption and by \cref{6.10inPaper}, every vertex in $S_1(a)$ is forward facing with respect to $a$.

    Second we have $x_{j-1},x_{j+1}\in S_1(x_j)$ because $x_{j-1},x_j,x_{j+1}$ is a path by assumption.
    
    Third we observe $x_{j-1},x_{j+1}\in S_1(a)\cup S_2(a)$ and $S_1(a)$ is totally disconnected because $\mathcal{C}\Sigma_{0,5}$ has no triangles. Now $x_{j-1},x_{j+1}$ are adjacent to $x_j\in S_1(a).$ Thus, $x_{j-1},x_{j+1}\not\in S_1(a)$ so $x_{j-1},x_{j+1}\in S_2(a).$ This verifies the conditions of \cref{6.14inPaper}.
\end{proof}
Now we have the tools to construct the preliminary path as stated in \cref{modifiedPath}.
\begin{proof}[Proof of \cref{modifiedPath}]
By \cref{6.13inPaper}, $S_1(a)\cup S_2(a)$ is connected. Since $b,b'\in S_1(a)$ this implies there exists a path $b=x_0,...,x_l=b'$ contained in $S_1(a)\cup S_2(a).$ Now for each $x_j\in (S_{r-1}\cup S_{r})\cap S_1(a)$ replace the path segment $x_{j-1},x_j,x_{j+1}$ with the path $x_{j-1}=x_{j}^0,x_j^1,...,x_j^k=x_{j+1}$ for some $k\geq 0$ given by \cref{avoidS1}. Call this path $\gamma$. First we observe by construction that $\gamma$ has no vertex in $(S_{r-1}\cup S_r)\cap S_1(a).$

Now we check that $\gamma$ satisfies the conclusions of \cref{modifiedPath} with each following sublemma:

\begin{sublemma}\label{condition0}
    The path $\gamma$ is contained in $S_1(a)\cup S_2(a)\cup S_3(a).$
\end{sublemma}
\begin{proof}
    By construction the vertices in $\gamma$ are either in $S_1(a)\cup S_2(a)$ or in $(S_2(a)\cup S_3(a))\cap B_2(x_j)$ for some $j\leq l.$
\end{proof}

\begin{sublemma}\label{condition1}
    If $v$ is a vertex in $\gamma$ and $v\in S_3(a),$ then $d(v,(S_{r-1}\cup S_r)\cap S_1(a))\leq 2.$
\end{sublemma}
This establishes part (1) of \cref{modifiedPath}.
\begin{proof}
    The original path $b=x_0,...,x_l=b'$ is contained in $S_1(a)\cup S_2(a)$ so if $v\in S_3(a),$ then $v$ must have been obtained from replacing the segment $x_{j-1},x_j,x_{j+1}$ with the path $x_{j-1}=x_{j}^0,x_j^1,...,x_j^k=x_{j+1}$ for some $k\geq 0.$ In particular, $v=x_j^i$ for some $i\leq k.$ By \cref{avoidS1}, $v=x_j^i\in (S_{2}(a)\cup S_3(a))\cap B_2(x_j)$ so $d(v,x_j)\leq 2.$ Additionally, $x_j\in (S_{r-1}\cup S_r)\cap S_1(a).$ Thus $d(v,(S_{r-1}\cup S_r)\cap S_1(a))\leq 2.$ 
\end{proof}

\begin{sublemma}\label{condition2}
    The only vertices in $\gamma$ which are in $S_1(a)$ are also in $S_{r+1}.$
\end{sublemma}
This establishes part (2) of \cref{modifiedPath}.
\begin{proof}
    Let $v$ be a vertex on $\gamma$ such that $v\in S_1(a).$ Now $a\in S_r$ so $v\in S_{r-1}\cup S_r\cup S_{r+1}.$ But by construction $\gamma$ has no vertices in $(S_{r-1}\cup S_r)\cap S_1(a)$ because any such vertices in the original path were replaced by a path in $S_2(a)\cup S_3(a).$ Thus, $v\in S_{r+1}.$ 
\end{proof}

Since we have verified the conclusions of \cref{modifiedPath} for arbitrary $a\in S_r$ and $b,b'\in S_1(a)\cap S_{r+1},$ this finishes the proof.
\end{proof}


\subsection{Pushing the path up}\label{subsection:pushing_the_path_up}  

Now we will apply Dehn twists to the path obtained in \cref{modifiedPath} to make sure it lies in $S_r \cup S_{r + 1}$. 

We first fix some important notations.  

\begin{remark}\label{rem:rem_on_Dehn_twists} 
    Suppose $a,b \in \mathcal{C} \Sigma_{0,5}$. Let $T_a (b)$ denote the left Dehn twist of $b$ around $a$. Henceforth, we will refer to left Dehn twists as just Dehn twists. 
    
    In addition, we use $d_a$ to denote the distance between the projections to the curve graph of the annular subsurface associated to an element $a$ of $\mathcal{C} \Sigma_{0,5} $. 
\end{remark}

We will make use of the following basic fact. 
\begin{proposition}\label{prop:key_fact} 
Suppose $a,b$ are vertices in $\mathcal{C} \Sigma_{0,5} $ such that $d (a,b) \geq 2$. Then 
\begin{equation}
\lim_{N \to \infty} d_a (b, \dtwist{N}{a} (b))= \infty  . 
\end{equation} 
\end{proposition}    

\begin{lemma}\label{lem:cor_of_bgi} 
Suppose $a \in S_r$ and $d (b,a) \geq 2$. Then there exists a positive integer $N (a,b)$, such that for all $N ^\prime  \geq N (a,b)$, we have $d_a (\dtwist{N ^\prime }{a} (b),c) \gg M$.
\end{lemma}

\begin{proof}
For all integers $m$, we have
\begin{equation}
d_a (\dtwist{m}{a} (b),c) \geq d_a (\dtwist{m}{a} (b),b) - d_a (b,c), 
\end{equation}
where $d_a (b,c)$ is a constant. Thus, the lemma follows from \cref{prop:key_fact}. 
\end{proof}

\begin{remark}
    For the rest of \cref{subsection:pushing_the_path_up}, we will continue to use $N (a,b)$ to denote the constant in \cref{lem:cor_of_bgi}. Note that $N(a,b)$ depends on $a,b$. 
\end{remark}

\begin{corollary}\label{lem:Dehn_twist_lifts} 
Suppose $a \in S_r$ and $d (b,a) \geq 2$. If $N ^\prime  \geq N (a,b)$, then 
\begin{equation}
d (c, \dtwist{N ^\prime }{a} (b)) \geq r. 
\end{equation}
\end{corollary}
\begin{proof}
By \cref{lem:cor_of_bgi} and \cref{BGI}, any geodesic from $\dtwist{N ^\prime }{a} (b)$ to $c$ must contain a vertex that lies in $B_1 (a)$. This implies that $d (\dtwist{N ^\prime }{a} (b),c) \geq r$. 
\end{proof}


\subsection{Main lemma} 

\begin{lemma}\label{lem:crucial_lemma}
Suppose $a \in S_r$ and $b,b ^\prime \in S_{ r + 1}\cap S_1 (a)$. Then there exists a path $b,x_1 , \ldots ,x_{ l }, b ^\prime$ with four properties: 
\begin{enumerate}
\item $1 \leq d (x_i, a) \leq 3$. 
\item $r \leq d (x_i, c) \leq r + 2$. 
\item If $d (x_i,c ) = r$, then $d (x_i,a) = 2$, there exists a unique vertex $z$ adjacent to both $x_i$ and $a$, $z \in S_{r - 1}$, and $z$ is the unique backtrack of $x_i$. 
\item If $d (x_i,c) = r$ and if $a$ has unique backtracking, then $x_i$ has no sidestepping. 
\end{enumerate}
\end{lemma}

\begin{remark}
    This lemma improves \cite[Lemma 6.16]{Wright1} in that our lemma also shows that $d (x_i,c) \leq r + 2$. 
\end{remark}

\begin{proof}
We first construct a path and then prove that it satisfies the four listed properties. 

We begin by considering the path $\alpha$ that \cref{modifiedPath} gives us. Let $b, y_1 , \ldots ,y_l, b ^\prime $ be the vertices of the path. By \cref{lem:cor_of_bgi}, for all $i$ such that $d (y_i,a ) \geq 2$, there exists a positive integer $N (a,y_i)$ such that if $N ^\prime  \geq N (a, y_i)$, then $d_a (\dtwist{N ^\prime }{a} (y_i), c) \gg M$. Take $N = \max_i (N (y_i,a))$. Let $\gamma$ be the path obtained by applying $\dtwist{N}{a}$ to $\alpha$. The vertices of $\gamma$ are then
\begin{equation}
    b, \dtwist{N}{a} (y_1) , \ldots ,\dtwist{N}{a}(y_{l}), b ^\prime. 
\end{equation}
Let $x_i = \dtwist{N}{a} (y_i)$ for all $1 \leq i \leq l$.  

Proposition \ref{modifiedPath}, as well as the fact that Dehn twists preserve distance (\cref{rem:rem_on_Dehn_twists}), verifies property (1) above. 

Now we verify property (2). We first claim that for all $i$, $d (x_i, c) \geq r$. Let us fix some $i$. If $d (y_i,a) \geq 2$, then \cref{lem:Dehn_twist_lifts} implies that $d(x_i, c) = d (\dtwist{N}{a} (y_i), c) \geq r$. On the other hand, $d(y_i,a)\geq 1$ by construction of $\alpha.$ So the only remaining case to consider is if $d (y_i,a ) = 1$, then the assumptions on the path $\alpha$ imply that $y_i \in S_{ r + 1}$. So $d (x_i, c ) = d (T^N_a (y_i),c)= d (y_i,c) \geq r$. 

Next, we claim that for all $i$, $d (x_i, c) \leq r + 2$. This follows from the observation that if $y_i \in S_3 (a)$, then by assumptions on the path $\alpha$, there exists $z_i \in S_1 (a) \cap  (S_r \cup S_{ r - 1})$ such that $d (y_i, z_i) \leq 2$. But since Dehn twists preserve distances and fix vertices adjacent to the center of the twist, 
\begin{equation}
    d (x_i, z_i) = d (\dtwist{N}{a} (y_i), \dtwist{N}{a}  (z_i))  = d (y_i, z_i ). 
\end{equation} 
And so $d (x_i, z_i)\leq 2$. So 
\begin{equation}
    d (x_i,c) \leq d (x_i, z_i) + d (c, z_i) \leq r + 2. 
\end{equation}
This finishes the verification of property (2). 

To verify property (3), we suppose $d (x_i,c) = r$. Recall that by definition, $x_i = \dtwist{N}{a} (y_i)$. If $d (y_i , a ) = 1$, then by construction of $\alpha$, we have $y_i \in S_{r + 1}$. Since $\dtwist{N}{a}$ fixes $y_i$, we conclude that $x_i = \dtwist{N}{a} (y_i)$ belongs to $ S_{r + 1}$. This contradicts the assumption that $d (x_i,c) = r$. So we must have $d (y_i,a) \geq 2$. 

And so by \cref{lem:cor_of_bgi} and \cref{BGI}, every geodesic from $x_i = \dtwist{N}{a} (y_i)$ to $c$ must pass through $B_1 (a)$. Let $\zeta$ be one such geodesic and $z$ be one vertex in $\zeta \cap  B_1 (a)$. Since $d (x_i,a) \geq 2$, $z $ must belong to $ S_{ r - 1}$, implying that $d (x_i,a)= 2$. By construction, $z$ is a vertex adjacent to both $x_i$ and $z$. It is the unique such vertex because $\mathcal{C} \Sigma_{0,5} $ has no quadrilaterals. 

To finish verifying property (3), it remains to show that $z$ is the unique backtrack of $x_i$. Let $z ^\prime $ be a backtrack of $x_i$. There is a geodesic $\tilde{\zeta}$ connecting $z$ to $c$ that passes through $z ^\prime $. By the Bounded Geodesic Image Theorem, $\tilde{\zeta}$ must intersect $B_1 (a)$. Since $z ^\prime \in S_{r - 1}$, $z ^\prime $ must in fact belong to $B_1 (a)$. Because $\Sigma_{0,5} $ has no quadrilaterals, $z$ and $z ^\prime $ must coincide. This verifies property (3). 

To verify property (4), assume $a$ has unique backtracking and $x_i \in S_r$. Suppose for the sake of contradiction that $s$ is a sidestep of $x_i$. We note that $s$ is not adjacent to $a$ because otherwise $x_i,s,a,z$ would form a quadrilateral, a contradiction. $s$ is also not equal to $a$, since otherwise $x_i,a,z$ form a triangle, a contradiction.  

Let $z$ be the unique neighbor of $x_i$ and $a$ constructed during the verification of property (3). During the verification of property (3), we proved that $d (y_i,a ) \geq 2$. So by \cref{lem:cor_of_bgi}, $d_a (x_i,c) \gg M$. Additionally, since $d (x_i, s)=1$, by the coarse-Lipschitz property of $d_a$, we have $d_a (x_i,s)$ is bounded. So by the triangle inequality, $d_a (s,c) \gg M$. By \cref{BGI}, we know that every geodesic from $s$ to $c$ passes through $B_1 (a)$. 

Let $\eta$ be one such geodesic. Since $s \in S_r$ and $s$ is not adjacent or equal to $a$, we have $\eta \cap B_1 (a) \subset S_{r - 1}$. But since $a$ has unique backtracking, the only vertex in $B_1 (a) \cap S_{r - 1}$ is $z$. This shows that $\eta $ must pass through $z$. But then $s,z,x_i$ form a triangle, a contradiction. This proves property (4).
\end{proof}


\subsection{Proving \cref{mainPropLowComp}} Before we begin the proof of \cref{mainPropLowComp}, we will need to make use of the following lemmas. 

\begin{lemma}\label{pushUpUBT}
    Suppose $a \in S_r$ has unique backtracking and $b, b' \in S_{r+1}$ are both adjacent to $a$. Then there exists a path from $b$ to $b'$ entirely in $(S_{r+1} \cup S_{r+2}) \cap B_4(a)$.
\end{lemma}
\begin{proof}
    Consider the path from $b$ to $b'$ given by \cref{lem:crucial_lemma}. Each vertex on this path that lies in $S_r$ is forward facing and also in $B_2(a)$. Forward facing vertices have no side stepping, so this path has no adjacent vertices in $S_r$. Thus we can apply \cref{6.14inPaper} to each vertex in $S_r$ to obtain the appropriate path in $(S_{r+1} \cup S_{r+2}) \cap B_4(a).$
\end{proof}

\begin{lemma}\label{pushUpA} Suppose $a \in S_r$ and $b, b' \in S_{r+1}$ are both adjacent to $a$. Then there exists a path from $b$ to $b'$ entirely in $(S_{r+1} \cup S_{r+2}) \cap B_6(a)$.
\end{lemma}
\begin{proof}
     \cref{lem:crucial_lemma} gives a path from $b$ to $b'$ in $(S_r \cup S_{r+1} \cup S_{r+2}) \cap B_3(a)$ such that each vertex on this path that lies in $S_r$ has unique backtracking and is in $B_2(a)$. By \cref{6.5inPaper} we can modify the path at each pair of adjacent vertices that lie in $S_r$ to obtain a new path in $(S_r \cup S_{r+1} \cup S_{r+2}) \cap B_4(a)$ with the additional assumption that no two adjacent vertices are in $S_r$. Now we can apply \cref{pushUpUBT} to each vertex in $S_r$ to obtain the appropriate path in $(S_{r+1} \cup S_{r+2}) \cap B_6(a)$. 
\end{proof}

Next, we want to recall \cite[Lemma 2.1]{Wright1} for the sufficient conditions for connectivity of spheres:
\begin{lemma}\label{suffCondConnSpheres}\cite[Lemma 2.1]{Wright1}
    Let $\Gamma$ be an arbitrary graph and fix $c \in \Gamma$. Fix $w > 0$, and let $r > 0$ be arbitrary. Suppose the following conditions hold:
    \begin{itemize}
        \item[(1)] For every $z \in S_r(c)$ and $x, y \in S_{r+1}(c) \cap B_1(z)$ there exists a path 
        $$
        x = x_0, x_1, \dots, x_l = y
        $$
        with
        $$
        x_i \in S_{r+1}(c) \cup \dots \cup S_{r+w}(c)
        $$
        for $0 \leq i \leq l$. 
        \item[(2)] For every adjacent pair $x,y \in S_r(c)$ there exists a path 
        $$
        x = x_0, x_1, \dots, x_l = y
        $$
        with 
        $$
        x_i \in S_{r+1}(c) \cup \dots \cup S_{r+w}(c)
        $$
        for $0<i<l$.
    \end{itemize}
    Then $S_r(c) \cup S_{r+1}(c) \cup \dots \cup S_{r+w-1}(c)$ is connected.
\end{lemma}

\begin{lemma}\label{conditions}
    In $\mathcal{C}\Sigma_{0,5}$, for all $r \geq 0$ the following hold:
    \begin{itemize}
        \item [(1)] For every $z \in S_r$ and $x, y \in S_{r+1} \cap B_1(z)$ there exists a path 
        $$
        x = x_0, \dots, x_l = y
        $$
        with
        $$
        x_i \in (S_{r+1} \cup S_{r+2}) \cap B_6(z)
        $$
        for $0 \leq i \leq l$.
        \item [(2)] For every adjacent pair $x,y \in S_r$ there exists a path 
        $$
        x = x_0, x_1, x_2, x_3, x_4 = y
        $$
        with
        $$
        x_i \in S_{r+1} \cup S_{r+2}
        $$
        for $0 < i < 4$.
    \end{itemize}
\end{lemma}

\begin{proof} The first claim is \cref{pushUpA} and the second claim is \cref{6.5inPaper}.
\end{proof}

\begin{proof}[Proof of \cref{mainPropLowComp}]Since the curve graphs $\mathcal{C}\Sigma_{0,5}$ and $\mathcal{C}\Sigma_{1,2}$ for the low complexity surfaces are isomorphic, it suffices to prove \cref{mainPropLowComp} for $\mathcal{C}\Sigma_{0,5}$. The result follows immediately from combining \cref{suffCondConnSpheres} and \cref{conditions}. 
\end{proof}


\section{Medium complexity}


Throughout this section we assume $\Sigma$ is medium complexity. Again we fix a center vertex $c$ and let $S_r = S_r(c)$. In this section we upgrade the results from \cite[Theorem 1.1]{Wright1} to prove \cref{mainPropMedComp}: $S_r$ is connected for medium complexity surfaces. 

\subsection{Organization}
We use \cite[Theorem 1.1 (2)]{Wright1} that $S_r\cup S_{r+1}$ is connected and begin with a path in $S_r\cup S_{r+1}$. Then we use the definition $\mathcal{O}(z)$, introduced by Wright, as a tool to push the path into $S_{r+1}$ by allowing the path to contain vertices which need not be essentially non-separating.


\subsection{Essentially non-separating curves} 

\begin{definition}
    A curve on $\Sigma$ is called a \emph{pants curve} if it bounds a genus 0 subsurface with 2 punctures. 
\end{definition}

\begin{definition}\label{essentially_non_separating}
    A curve on $\mathcal{C}\Sigma$ is \emph{essentially non-separating} if it is non-separating or a pants curve. A two-component multi-curve $\alpha \cup \beta$ is \emph{essentially non-separating} if $\alpha$ and $\beta$ themselves are essentially non-separating, and either
    \begin{itemize}
        \item[(1)] $\alpha \cup \beta$ is non-separating, 
        \item[(2)] at least one of $\alpha$ or $\beta$ is a pants curve, or 
        \item[(3)] $\alpha \cup \beta$ bounds a genus 0 subsurface with 1 puncture. 
    \end{itemize}
\end{definition}

For $c\in\mathcal{C}\Sigma$, we can define $\mathcal{C}_c\Sigma$ as the subgraph of $\mathcal{C}\Sigma$ whose vertex set is $\{c\}$ union all essentially non-separating curves on $\mathcal{C}_c\Sigma$. Disjoint curves $\alpha$ and $\beta$ are joined by an edge if either $\alpha \cup \beta$ is essentially non-separating or they have different distances to $c$. 

To fix notation, let $S_r^c = S_r \cap \mathcal{C}_c\Sigma$. 
\begin{remark}\label{defCoincides}\cite[Lemma 5.2]{Wright1}
Wright showed that $S_r^c$ coincides with the sphere of radius $r$ in $\mathcal{C}_c\Sigma$. 
\end{remark}

We now recall the following results:




\begin{lemma}\label{ENSunionConnected}
    $S_r^c \cup S_{r+1}^c$ is connected.
\end{lemma} 
\begin{proof}
     \cite[Proposition 5.4]{Wright1} verifies that the sufficient conditions for the connectivity of spheres in \cref{suffCondConnSpheres} hold in $\mathcal{C}_c\Sigma$ with $w=2$.
\end{proof}


\begin{lemma}\label{connectToSrc}\cite[Lemma 5.3]{Wright1}
    Suppose $\Sigma$ has medium complexity. For all $x\in S_r,$ then either $x\in S_r^c$ or there exists $x'\in S_r^c\cap S_1(x).$
\end{lemma}



\subsection{Definition and properties of $\mathcal{O}(z)$}

In order to prove \cref{mainPropMedComp}, we make use of the following definition and prove several of its properties. 

\begin{definition}
    For any $z \in S_r^c$, define
$$
\mathcal{O}(z) = \{a \in S_1(z) \cap \mathcal{C}_c\Sigma : d_U(a, c) > M\} 
$$
where $U$ is the unique component of $\Sigma - z$ that is not a pants. Observe that $\mathcal{O} (z) \subseteq S_{r+1}^c$. 
\end{definition}

Recalling \cite[Lemma 7.2]{Wright1}, we know we can connect any essentially non-separating curve to $\mathcal{O}(z)$:
\begin{lemma}\label{connectToOz}\cite[Lemma 7.2]{Wright1}
    Let $z \in S_r^c$ and $U$ be the unique connected component of $\Sigma - z$ that is not a pants. Then for all $N>0,$ any $x \in S_1(z) \cap S_{r+1}^c$ can be connected to some $e \in \mathcal{O}(z)$ by a path in $S_1(z) \cap S_{r+1}^c$. Moreover, $e$ can be taken such that $d_U(e, c) > N$.
\end{lemma}

Additionally, we will make use of the following lemma:
\begin{lemma}\label{connectTwoPointsInOz}
    Let $z \in S_r^c$ and $a,b \in \mathcal{O}(z)$. Then $a, b$ can be connected by a path contained entirely in $S_{r+1}$. 
\end{lemma}
\begin{proof}
    Let $U$ be the unique connected component of $\Sigma - z$ that is not pants. Observe that the subsurface projection $\rho_U(c)$ is a finite set with diameter bounded by some constant $k$ (see section 2). Thus there exists $c' \in \rho_U(c)$ such that $d_{\mathcal{C}\Sigma}(c, c') \leq k$. Since $a, b \in \mathcal{O}(z)$, both $d_U(a, c), d_U(b ,c) \geq M + 1$, so by the triangle inequality, 
\begin{equation}
        a, b \in \bigcup_{r'=M+1+k}^{\infty}S_{r'}(c'),
\end{equation}
    where each $S_{r'}(c')$ is a sphere in $\mathcal{C}U$. This union is a subgraph of $\mathcal{C}U$. It is connected because $U$ is low complexity, and so \cref{mainPropLowComp} gives that $S_{M+1+k}(c') \cup S_{M+2+k}(c')$ is a connected subset of $\mathcal{C} U$. Thus we can find a path in $\mathcal{C}U$
    $$
    a = p_0, \dots, p_l = b
    $$
    such that each $d_U(p_i, c') \geq  M+1+k$. Then by the triangle inequality, $d_U(p_i,c)>M.$

Applying \cref{BGI} for all $0 < i < l$, every geodesic from $p_i$ to $c$ must go through $z$, as $z$ is the only vertex not cutting $U$ since it is essentially non-separating. By construction, for all $i$, $p_i$ lies entirely within $U$ and so $d(z, p_i) = 1$. Since $d(z, c) = r$ and , $d(p_i, c) = r + 1$ for all $i$, as desired.  
\end{proof}


\subsection{Proving \cref{mainPropMedComp}}

\begin{proof}[Proof of \cref{mainPropMedComp}]
Suppose $x,y\in S_{r+1}$ are arbitrary. By \cref{connectToSrc} we can connect $x,y$ to $x',y'\in S_{r+1}^c$ respectively, so it suffices to find a path connecting $x',y'$ inside $S_{r+1}.$ By \cref{ENSunionConnected}, $S_r^c\cup S_{r+1}^c$ is connected, so there exists a path $x'=x_0,x_1,...,x_k=y'$ contained in $S_r^c\cup S_{r+1}^c$. 

We now make use of a sublemma:
\begin{sublemma}\label{noBadEdgesSublemma}
        The path from $x'$ to $y'$ above can be taken to have no two consecutive vertices in $S_r^c.$
\end{sublemma}
\begin{proof}
        This follows from \cite[Lemma 5.4, part (2)]{Wright1} that for each $x_i,x_{i+1}\in S_r^c$, there exists a path $x_i=x_i^0,x_i^1,x_i^2=x_{i+1}$ such that $x_i^1\in S_{r+1}^c.$
\end{proof}

By \cref{noBadEdgesSublemma}, for each vertex $x_i$ in the path from $x'$ to $y',$ if $x_i\in S_r^c,$ then both $x_{i-1}$ and $x_{i+1}$ must be in $S_{r+1}^c.$ In particular, since $x_{i-1},x_i,x_{i+1}$ is a path, we have $x_{i-1},x_{i+1}\in S_1(x_i)\cap S_{r+1}^c.$ 

Now applying \cref{connectToOz}, to $x_i,$ there exists $x_{i-1}'$ and $x_{i+1}'$ in $\mathcal{O}(x_i)$ which can be connected to $x_{i-1}$ and $x_{i+1}$ respectively with paths contained in $S_1(x_i)\cap S_{r+1}^c$ such that $d_U(x_{i-1}',c)\gg M$ and $d_U(x_{i+1}',c) \gg M.$ 

Applying \cref{connectTwoPointsInOz}, we can connect $x_{i-1}'$ and $x_{i+1}'$ by a path entirely in $S_{r+1}.$ Thus, for consecutive vertices $x_{i-1},x_i,x_{i+1}$ in the path from $x'$ to $y'$ where $x_{i-1},x_{i+1}\in S_{r+1}^c$ and $x_i\in S_r^c$, we can remove $x_i$ and connect $x_{i-1}$ to $x_{i+1}$ by a path contained in $S_{r+1}.$ By \cref{noBadEdgesSublemma} no two consecutive vertices in the path were in $S_r^c,$ so this construction eliminates all vertices in $S_r$ and results in a path from $x'$ to $y'$ contained in $S_{r+1}$, as desired.
\end{proof}

\section{Structure of $S_2$ in Low Complexity}\label{sec:structure_of_S2}

The main aim of this section is to prove \cref{thm:S2_connected}. Throughout this section we will work with the low complexity surface $\Sigma = \Sigma_{0,5} $. During the proof, we will also show that in $\mathcal{C} \Sigma_{0,5} $, the sphere $S ^\prime _2 (c)$ has the structure of a $\mathbb{Z}$-bundle over $S_1 (c)$.

\subsection{Basic Definitions}   

We begin with two basic definitions. 

\begin{definition}\label{defn:fiber} 
Suppose $x \in S_1 (c)$. Consider $S_1 (x)$, which is a copy of the Farey graph, in which $c$ is a vertex. Let $E_{x}$ denote the subset of $S_1 (x)$ that has Farey distance 1 from $c$. In other words, $E_x$ consists of curves that are disjoint from $x$ and have intersection number 2 with $c$. 
\end{definition}

\begin{definition}\label{defn:projection} 
Let $v \in S_2 (c)$ be any vertex. Define $\beta (v)$ as the unique backtrack of $v$ in $S_1 (c)$. In other words, $\beta(v)$ is the unique vertex adjacent to both $v$ and $c$.
\end{definition}

\begin{proposition}\label{prop:chop_down}
Suppose $v \in S_2 (c)$. If $v$ is a non-isolated vertex in $S_2 (c)$, then $v \in E_{ \beta (v)}$.   
\end{proposition}
\begin{proof}
Since $v$ is nonisolated in $S_2 (c) $, $v$ is adjacent to some $w \in S_2 (c)$. Since $\mathcal{C} \Sigma_{0,5} $ has no triangles, $\beta (w) \neq \beta (v)$. So $c, \beta (v), v, w, \beta (w)$ is a cycle of length 5. Since all cycles of length 5 in the low complexity curve graph are pentagons (\cite[Theorem 3.1]{FiniteRigid}), $c, \beta (v), v, w, \beta (v)$ is a pentagon. Thus, $v$ is Farey adjacent to $c$ in $S_1 (\beta (v))$, proving that $v \in E_{ \beta (v)}$.  
\end{proof}

\begin{remark}\label{rem:composition_of_delta}

\cref{prop:chop_down} implies that $S ^\prime _2 (c) = \bigsqcup_{x \in S_1 (c) } {E_{x} }$.  So the map $\functionheader{\beta \vert _{S ^\prime _2 (c)} }{S ^\prime _2 (c)}{S_1 (c) }$ gives a fiber bundle. We will refer to the map $\beta \vert _{S ^\prime _2 (c)}$ as just $\beta$.   
\end{remark} 

We now give the above fiber bundle the additional structure of a $\ints$-bundle. We begin by recalling the notion of a half Dehn twist. 

\begin{notation}\label{Half_Dehn_twist} 
Let $v \in \mathcal{C} \Sigma_{0,5} $ be any vertex. Then we let $\tau_v$ denote the half (right) Dehn twist around $v$. Furthermore, we let $H_v$ denote the infinite cyclic group generated by $\tau_v$ (viewed as a subgroup of the mapping class group of $\Sigma_{0,5} $).  
\end{notation}

\begin{fact}\label{lem:Dtwist_transitive}
Suppose $x \in S_1 (c)$. Then $H_c$ acts on $E_x$ simply transitively. Indeed, the set of vertices adjacent to $x$, which includes $c$ and $E_x$, can be naturally identified with the Farrey graph, and $H_c$ acts simply transitively on the set of vertices adjacent to $c$ in this Farrey graph. 
\end{fact}

This fact implies that the $H_c$-action makes the $\functionheader{\beta}{S ^\prime _2 (c)}{S_1 (c) }$ into a $\ints$-bundle, as we now make explicit.

\begin{remark}\label{rem:explicit_id} 
For all $y \in S_1 (c)$, we fix for the rest of this section some arbitrary $\overline{y} \in E_{y}$. Then there is an explicit bijection from $\ints$ to $E_y$ given by $n \mapsto \tau_c ^{n} (\overline{y})$. Let $\zeta_y$ denote the inverse of this bijection (so $\zeta_y$ maps $E_y$ to $\ints$). 
\end{remark}

\subsection{Perfect pairing between some of the fibers}
In this subsection, we show that if $x_1,x_2 \in S_1 (c) $ and $x_1,x_2$ are Farey connected, then $E_{x_1}, E_{x_2}$ have a ``perfect pairing," which we will make precise below. We first introduce a piece of notation. 

\begin{definition}
    Suppose $x_1,x_2$ are in $\in S_1 (c) $ and $x_1,x_2$ are Farey connected. Then let $\mathcal{E} (x_1, x_2)$ denote the set of all edges in $\mathcal{C} \Sigma_{0,5} $ with one vertex in $E_{x_1}$ and another vertex in $E_{x_2}$. 
\end{definition}

The following proposition explains how $\mathcal{E} (x_1,x_2)$ gives a ``perfect pairing" between $E_{x_1}$ and $E_{x_2}$. 

\begin{proposition}\label{prop:perfect_matching}
Suppose $x_1,x_2 \in S_1 (c)$ and $x_1$ is Farey connected to $x_2$. Then there exists a bijection $\functionheader{\psi}{E_{x_1}}{E_{ x_2}}$ such that 
\begin{equation}
    \mathcal{E} (x_1,x_2) = \setbuilder{\enum{v,\psi (v)}}{v \in E_{x_1}}. 
\end{equation}
\end{proposition}

In other words, the proposition says that every vertex of $E_{x_1}$ is joined by an edge to a unique vertex of $E_{x_2}$, and vice versa. The bijection is such that for all $v \in E_{x_1} $, $\psi(v)$ is the unique element of $E_{x_2}$ joined to $v$ by an edge.

\begin{proof}
Since $x_1$ is Farey connected to $x_2$, by \cite[Lemma 6.6]{Wright1}, there exists $s_1, s_2 \in S ^\prime _2 (c)$ such that $c, x_1, s_1, s_2, x_2$ is a pentagon. By definition of a pentagon, $s_1 \in E_{ x_1} $ and $s_2 \in E_{ x_2}$. Applying all integer powers of the half twist $\tau_c$ to the edge $\enum{s_1,s_2}$, we get a collection of edges $$\setbuilder{\enum{\tau^{n}_c (s_1), \tau^{n}_c (s_2)}}{n \in \ints}.$$ Call the collection $\Omega$. 

By \cref{lem:Dtwist_transitive}, if $e_1,e_2 \in \Omega$, then $e_1, e_2$ share no vertices. Also by \cref{lem:Dtwist_transitive}, each vertex in $E_{x_1}$ is contained in an edge $\Omega$ and likewise each vertex in $E_{x_2}$ is contained in an edge $\Omega$. These two facts guarantee the existence of a bijection $\functionheader{\psi}{E_{x_1}}{E_{x_2}}$ such that $\Omega = \setbuilder{\enum{v, \psi (v)}}{v \in E_{x_1}}$.   

Now it remains to verify that $\Omega= \mathcal{E} (x_1, x_2)$. It is clear that $\Omega \subset \mathcal{E} (x_1, x_2)$. To prove the converse, first observe that any edge $e \in \mathcal{E} (x_1, x_2)$ forms a pentagon with the vertices $x_1, x_2, c$. We know that all the pentagons containing $x_1,x_2,c$ are obtained from our initial pentagon $\enum{c, x_1, s_1, s_2, x_2}$ by applying a power of $\tau_c$ (because given any two pentagons, there is a mapping class taking one to the other, and if this mapping class fixes $x_1, x_2, c$, it must be a power of $\tau_c$). Hence, $e$ is obtained by applying a power of $\tau_c$ to the edge $\enum{s_1, s_2}$, and so $e \in \mathcal{E} (x_1, x_2)$. This shows that $\mathcal{E} (x_1, x_2) \subset \Omega$, and hence proves the proposition. 
\end{proof}

This ``perfect pairing" between $E_{x_1}$ and $E_{x_2}$ (for all Farey connected $ x_1,x_2 \in S_1 (c) $) that we just found is compatible with the action of $H_c$ on $E_{x_1}$ and $E_{x_2}$. More precisely, we have the following.

\begin{corollary}\label{lem:equivariance} 
Suppose $x_1, x_2 \in S_{1}  (c)$ and $x_1,x_2$ are Farey connected. Let $\functionheader{\psi}{E_{x_1}}{E_{x_2}}$ constructed in \ref{prop:perfect_matching}. Then $H_c $ acts on $E_{x_1}, E_{x_2}$ $\psi$-equivariantly, i.e. for all $g \in H_c$ and all $v \in E_{x_1}$, we have 
\begin{equation}
\psi (gv) = g \psi (v). 
\end{equation}
\end{corollary} 
\begin{proof}
Define the set $\Omega$ as in the proof of \cref{prop:perfect_matching}. 

Suppose $g= \tau_c ^{m}$ and $v \in E_{x_1} $. We know that $\enum{v, \psi (v)} \in \Omega$. By construction of $\Omega$, we have $\enum{\tau_c ^{m} (v), \tau_c ^{m} (\psi (v))} \in \Omega$ as well. This implies that $\psi (\tau_c^{m} (v)) = \tau_c ^{m} (\psi (v))$. This proves the desired $\psi$-equivariance. 
\end{proof}

\subsection{Monodromy Number} 

In this subsection, we define the monodromy number associated to a Farey path in $S_1 (c)$. 

Suppose $x_1 , \ldots ,x_l$ all belong to $S_1 (c)$ and that they form a Farey path. Let $\psi_{i, i + 1}$, $\irange{1}{i}{l - 1}$,  be the bijections (between $E_{x_i}$ and $E_{ x_{i + 1}}$) obtained in \cref{prop:perfect_matching}. Choose any $v \in E_{x_1}$. By \cref{prop:perfect_matching}, we obtain a path in $S ^\prime  _2 (c) $ $$v , \psi_{ 12} (v),  \psi_{23} \psi_{ 12} (v), \ldots, \psi_{ (l - 1) l} \cdots \psi_{ 12} (v).$$ Using the identification of the two sets $E_{x_1 }, E_{x_l}$ with $\ints$ given by \cref{rem:explicit_id}, we compute an integer $\zeta_{x_l } (\psi_{ (l - 1) l} \cdots \psi_{ 12} (v)) - \zeta_{x_1} (v)$.     

\begin{proposition}\label{prop:mon_number_well_defined} 
For a fixed Farey path $\gamma$ as above, the number $\zeta_{x_l} (\psi_{ (l - 1) l} \cdots \psi_{ 12} (v)) - \zeta_{x_1} (v)$ is independent of the choice of $v \in E_{x_1}$.   
\end{proposition}

\begin{remark}\label{rem:choice_dependence} 
    If $x_l \neq x_1$ (i.e. our Farey path is not a Farey cycle), then the number $\zeta_{x_l} (\psi_{ (l - 1) l} \cdots \psi_{ 12} (v)) - \zeta_{x_1} (v)$ \emph{does} depend on the choices of $\overline{x}_l \in E_{x_l}$ and $\overline{x}_1 \in E_{x_1}$ that we made in \cref{rem:explicit_id} when we defined the bijections $\zeta_l $ and $ \zeta_1$ . 

    However, in the case $x_l = x_1$, then changing our choice of $\bar{x}_1$ would change $\zeta_{x_1} (\psi_{ (l - 1) 1} \cdots \psi_{ 12} (v)) $ and $ \zeta_{x_1} (v)$ by the same integer. Hence the number $\zeta_{x_1} (\psi_{ (l - 1) 1} \cdots \psi_{ 12} (v)) - \zeta_{x_1} (v)$ is independent of the choice of $\overline{x}_1$ that we made in \cref{rem:explicit_id}.  
\end{remark}

\begin{proof}[Proof of \cref{prop:mon_number_well_defined}]
If the path $\gamma$ has length 1, i.e. $l=2$, then the proposition follows from equivariance (\cref{lem:equivariance}). The general case follows from the case $l=1$.
\end{proof}

\begin{definition}\label{Monodromy Number}

Suppose $\gamma = x_1 , \ldots ,x_l$ is a Farey path in $S_1 (c)$. We call the number $\zeta_{x_l } (\psi_{ (l - 1) 1} \cdots \psi_{ 12} (v)) - \zeta_{x_1} (v)$ for some choice of $v$ the ``monodromy number" associated to $\gamma$. \cref{prop:mon_number_well_defined} shows that the monodromy number is independent of the choice of $v$. When $\gamma$ is a Farey cycle, by \cref{rem:choice_dependence}, the monodromy number is also independent of the choices made in \cref{rem:explicit_id}. 
\end{definition}

\subsection{Monodromy Number for a Triangle}\label{subsec:monodromy_number_of_triangle} 

In this subsection, we explicitly construct a Farey triangle in $S_1 (c)$ and calculate its monodromy number. 

\begin{remark}\label{rem:conventions} 
    For the rest of \cref{sec:structure_of_S2}, we fix two conventions for how we will pictorially represent $\Sigma_{0,5} $ and curves on it. First, we will label the five punctures on $\Sigma_{0,5}$ with elements of the set $\enum{1,2,3,4,5}$, as shown in \cref{fig:fundamental_triangle} and \cref{fig:monodromy_number}. Second, in these figures, we will represent an element $v \in \mathcal{C} \Sigma_{0,5} $ by an arc such that $v$ is the boundary of an $\varepsilon$-neighborhood of the arc.
\end{remark}

\begin{cons}\label{cons:fundamental_triangle} 

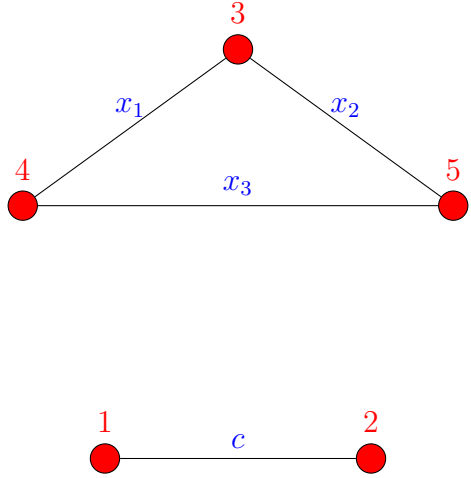
\begin{figure}
    \centering
    \begin{tikzpicture}
    \node[color=white,minimum size=6cm,draw,regular polygon,regular polygon sides=5] (p) {};
    \node[circle,draw,
        label=above:{$\textcolor{red}{3}$},
        fill=red] (3) at (p.corner 1) {};    
    \node[circle,draw,
        label=above:{$\textcolor{red}{4}$},
        fill=red] (4) at (p.corner 2) {};   
    \node[circle,draw,
        label=above:{$\textcolor{red}{1}$},
        fill=red] (1) at (p.corner 3) {}; 
    \node[circle,draw,
        label=above:{$\textcolor{red}{2}$},
        fill=red] (2) at (p.corner 4) {};
    \node[circle,draw,
        label=above:{$\textcolor{red}{5}$},
        fill=red] (5) at (p.corner 5) {};
    \path [-] (5) edge node[above] {$\textcolor{blue}{x_2}$} (3);
    \path [-] (3) edge node[above] {$\textcolor{blue}{x_1}$} (4);
    \path [-] (4) edge node[above] {$\textcolor{blue}{x_3}$} (5);
    \path [-] (2) edge node[above] {$\textcolor{blue}{c}$} (1);
    \end{tikzpicture}
    
    \caption{Fundamental Triangle}
    \label{fig:fundamental_triangle} 
\end{figure}

Let $c$ be the loop around punctures $1,2$ shown in \cref{fig:fundamental_triangle} (note that \cref{rem:conventions} is now in effect). We now construct a Farey cycle in $S_1 (c)$. Let $x_1$ (resp. $x_2$, $x_3$) be the loops around punctures $3,4$ (resp. punctures $3,5$, punctures $4,5$) also shown in \cref{fig:fundamental_triangle}. It is clear that $x_1 , x_2 , x_3 , x_1$ is a Farey cycle of length 3 in $S_1 (c) $. For the rest of \cref{subsec:monodromy_number_of_triangle}, we call this Farey cycle the ``fundamental triangle" and denote it by $\mathcal{T}$.    
\end{cons}

\begin{proposition}\label{prop:mon_number_of_triangle} 
The monodromy number of $\mathcal{T}$ is 1.  
\end{proposition} 
\begin{proof}
Let $\psi_{ 12}$ be the bijection between $E_{x_1}$ and $E_{ x_2}$ constructed in \cref{prop:perfect_matching}. Define $\psi_{23}$ and $\psi_{31}$ similarly. 

Let $v$ be the loop around punctures $2,5$ as shown in the \cref{fig:monodromy_number}. Then the loops $\psi_{12} (v), \psi_{23} \psi_{12} (v), \psi_{31} \psi_{23} \psi_{12} (v) $ must be the ones shown in the same figure. We see that $\psi_{31} \psi_{23} \psi_{12} (v) = \tau_c (v)$. As a result, we have 
\begin{equation}
\zeta_{x_1} (\psi_{31} \psi_{23} \psi_{12} (v))-  \zeta_{x_1} (v) = 1. \qedhere 
\end{equation}
\end{proof}

\begin{figure}[h!]
    \centering
    \begin{tikzpicture}
    \node[color=white,minimum size=6cm,draw,regular polygon,regular polygon sides=5] (p) {};
    \node[circle,draw,
        label=above:{$\textcolor{red}{3}$},
        fill=red] (3) at (p.corner 1) {};    
    \node[circle,draw,
        label=above:{$\textcolor{red}{4}$},
        fill=red] (4) at (p.corner 2) {};   
    \node[circle,draw,
        label=above:{$\textcolor{red}{1}$},
        fill=red] (1) at (p.corner 3) {}; 
    \node[circle,draw,
        label=above:{$\textcolor{red}{2}$},
        fill=red] (2) at (p.corner 4) {};
    \node[circle,draw,
        label=above:{$\textcolor{red}{5}$},
        fill=red] (5) at (p.corner 5) {};
    \path [-] (5) edge node[above] {$\textcolor{blue}{x_2}$} (3);
    \path [-] (3) edge node[above] {$\textcolor{blue}{x_1}$} (4);
    \path [-] (4) edge node[above] {$\textcolor{blue}{x_3}$} (5);
    \path [-] (2) edge node[above] {$\textcolor{blue}{c}$} (1);

    \definecolor{darkGreen}{RGB}{0, 130, 0}
    \path [-] (5) edge[yellow, line width=2pt] node[right] {$\textcolor{darkGreen}{v}$} (2);
    \path [-] (4) edge[yellow, line width=2pt] node[left] {$\textcolor{darkGreen}{\psi_{12}(v)}$} (1);
    \path [-] (1) edge[yellow, line width=2pt] node[above] {$\textcolor{darkGreen}{\psi_{31}\psi_{23}\psi_{12}(v)}$} (5);
    \path [-] (3) edge[yellow, line width=2pt, out=-63, in=-20, out looseness=1.9, in looseness=1.5] [bend left=99] node[right] {$\textcolor{darkGreen}{\psi_{23}\psi_{12}(v)}$} (2);
    \end{tikzpicture}
    \caption{Monodromy Number of the Fundamental Triangle} 
    \label{fig:monodromy_number} 
\end{figure}
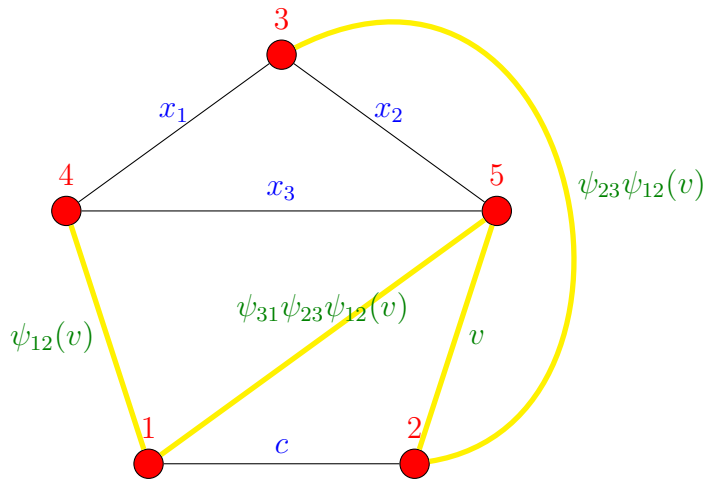

\begin{corollary}\label{cor:everything_to_one} 

Suppose $v,w \in E_{x_1}$, where $x_1$ is still the vertex defined in \cref{cons:fundamental_triangle}. Then $v$ can be connected to $w$ by a path in $S ^\prime _2 (c)$.  
\end{corollary}

\begin{proof}
We assume without loss of generality that $\zeta_{x_1}(w) - \zeta_{x_1} (v) = a > 0$. By \cref{prop:mon_number_of_triangle}, if $a=1$, then $v$ can be connected to $w$ by a path in $S_2 ^\prime  (c) $. The case of $a=1$ clearly implies the general case. 
\end{proof}

\subsection{Proving \cref{thm:S2_connected}} 

\cref{thm:S2_connected} now follows easily from \cref{prop:perfect_matching} and \cref{cor:everything_to_one}. 
\begin{proof}[Proof of \cref{thm:S2_connected}]
Fix some $v \in E_{x_1}$. Let $z \in S_1 (c) $ and $s \in E_{ z}$. It suffices to find a path in $S_2 ^\prime  (c) $ between $v$ and $s$. 

We first choose a Farey path $z=z_0, z_1 , \ldots , z_l = x_1$ contained in $S_1 (c) $. By applying \cref{prop:perfect_matching} $l$ times, we see $s$ is connected to some $s ^\prime \in E_{x_1}$ by some path in $S ^\prime _2 (c)$. By \cref{cor:everything_to_one}, $s ^\prime $ is connected to $v$ by some path in $S ^\prime _2 (c)$. This proves the theorem. 
\end{proof}
        
\bibliographystyle{amsalpha}
\bibliography{main}

\providecommand{\bysame}{\leavevmode\hbox to3em{\hrulefill}\thinspace}
\providecommand{\MR}{\relax\ifhmode\unskip\space\fi MR }
\providecommand{\MRhref}[2]{%
  \href{http://www.ams.org/mathscinet-getitem?mr=#1}{#2}
}
\providecommand{\href}[2]{#2}
\begin{thebibliography}{DDM13}

\bibitem[AL13]{FiniteRigid}
Javier Aramayona and Christopher~J. Leininger, \emph{Finite rigid sets in curve
  complexes}, J. Topol. Anal. \textbf{5} (2013), no.~2, 183--203.

\bibitem[BM15]{DeadEnds}
Joan~S. Birman and William~W. Menasco, \emph{The curve complex has dead ends},
  Geom. Dedicata \textbf{177} (2015), 71--74.

\bibitem[CH19]{chaika2019path}
Jon Chaika and Sebastian Hensel, \emph{Path-connectivity of the set of uniquely
  ergodic and cobounded foliations}, arXiv:1909.03668 (2019).

\bibitem[DDM13]{Spheres}
Spencer Dowdall, Moon Duchin, and Howard Masur, \emph{Spheres in the curve
  complex}, In the tradition of {A}hlfors-{B}ers. {VI}, Contemp. Math., vol.
  590, Amer. Math. Soc., Providence, RI, 2013, pp.~1--8.

\bibitem[Gab09]{Gabai1}
David Gabai, \emph{Almost filling laminations and the connectivity of ending
  lamination space}, Geom. Topol. \textbf{13} (2009), no.~2, 1017--1041.

\bibitem[Kla22]{klarreich2018boundary}
Erica Klarreich, \emph{The boundary at infinity of the curve complex and the
  relative {T}eichm\"{u}ller space}, Groups Geom. Dyn. \textbf{16} (2022),
  no.~2, 705--723.

\bibitem[LMS11]{LMS}
Christopher~J. Leininger, Mahan Mj, and Saul Schleimer, \emph{The universal
  {C}annon-{T}hurston map and the boundary of the curve complex}, Comment.
  Math. Helv. \textbf{86} (2011), no.~4, 769--816.

\bibitem[LS09]{LSconnectivity}
Christopher~J. Leininger and Saul Schleimer, \emph{Connectivity of the space of
  ending laminations}, Duke Math. J. \textbf{150} (2009), no.~3, 533--575.

\bibitem[MM00]{MasurMinsky}
Howard Masur and Yair Minsky, \emph{Geometry of the complex of curves {II}:
  Hierarchical structure}, 2000.

\bibitem[RS11]{RS}
Kasra Rafi and Saul Schleimer, \emph{Curve complexes are rigid}, Duke Math. J.
  \textbf{158} (2011), no.~2, 225--246.

\bibitem[Sch11]{SchleimerEnd}
Saul Schleimer, \emph{The end of the curve complex}, Groups Geom. Dyn.
  \textbf{5} (2011), no.~1, 169--176.

\bibitem[Web15]{WebbUniformBGI}
Richard C.~H. Webb, \emph{Uniform bounds for bounded geodesic image theorems},
  J. Reine Angew. Math. \textbf{709} (2015), 219--228.

\bibitem[Wri23]{Wright1}
Alex Wright, \emph{Spheres in the curve complex and linear connectivity of the
  {G}romov boundary}, 2023.

\end{thebibliography}

\end{document}